\documentclass[12pt,notitlepage]{amsart}
\usepackage{amsmath}
\usepackage{amscd}
\usepackage{amssymb}
\usepackage{enumerate}
\usepackage{indentfirst}

\newcommand{\Lcal}{\mathcal{L}}
\newcommand{\Ical}{\mathcal{I}}
\newcommand{\Ocal}{\mathcal{O}}

\newcommand{\Fcal}{\mathcal{F}}

\newcommand{\Kcal}{\mathcal{K}}
\newcommand{\Jcal}{\mathcal{J}}

\newcommand{\Gcal}{\mathcal{G}}
\newcommand{\Hcal}{\mathcal{H}}

\newcommand{\ol}{\overline}
\newcommand{\ox}{\otimes}
\renewcommand{\:}{\colon}

\newcommand{\IP}{\text{\bf P}}

\renewcommand{\d}{\partial}
\newcommand{\lra}{\longrightarrow}
\newcommand{\Div}{\text{\rm Div}}
\newcommand{\Cyc}{\text{\rm Cyc}}
\newcommand{\Twist}{\text{\rm Twist}}

\newtheorem{theorem}{Theorem}[section]
\newtheorem{lemma}[theorem]{Lemma}

\newtheorem{proposition}[theorem]{Proposition}

\theoremstyle{definition}

\newtheorem{example}[theorem]{Example}

\newtheorem{subsct}[theorem]{}
\theoremstyle{plain}

\font\smallrm=cmr8

\begin{document}
\author[{\smallrm Eduardo Esteves}]
{Eduardo Esteves}
\thanks{Supported by CNPq, Proc.~303797/2007-0 and 473032/2008-2, and FAPERJ, 
Proc.~E-26.102.769/2008.}
\title[{\smallrm Limits of Cartier divisors}]
{Limits of Cartier divisors}
\begin{abstract} Consider a one-parameter 
family of algebraic varieties degenerating to a 
reducible one. Our main result is a formula for 
the fundamental cycle of the limit subscheme of any 
family of effective Cartier divisors. The formula 
expresses this cycle as a sum of Cartier divisors, 
not necessarily effective, of the components of the 
limit variety.
\end{abstract}
\maketitle

\section{Introduction}

Consider a local one-parameter family of 
Noetherian schemes. More precisely, 
let $f\: X\to S$ be a flat 
map of Noetherian schemes, where $S$ stands for the 
spectrum of a discrete valuation ring. Let $s$ and 
$\eta$ denote the special and generic points of $S$; 
put $X_s:=f^{-1}(s)$ and $X_\eta:=f^{-1}(\eta)$. 
Assume that $X_s$ is of pure dimension and 
has no embedded components. 

Let $D$ be an effective Cartier divisor of $X$. 
View it as a subscheme of $X$, and 
let $\lim D$ be the schematic boundary of 
$D\cap X_\eta$. Then $\lim D\subseteq D\cap X_s$. 
Equality does not necessarily hold, as $D$ may 
contain components of $X_s$ in its support. 

This note presents a formula for the 
fundamental class $[\lim D]$ of $\lim D$ in terms of 
Cartier divisors of the components of $X_s$; see 
Theorem~\ref{thm}. The idea used in its proof is that, 
even though $D$ may not restrict 
to an effective Cartier divisor of a given component 
of $X_s$, a suitable modification of 
$D$ may. Suitable modifications may not exist. 
They do when $X_s$ is reduced, a consequence of 
Proposition~\ref{redcartier}. At any rate, when 
they exist, a formula for $[\lim D]$ is derived by 
keeping track of the modifications and their 
restrictions to the components of $X_s$.

The idea used in the proof of the main theorem is reminiscent of that 
behind the definition of limit linear series, 
as explained in \cite{EH}. 
And, in fact, the main application of the theorem 
so far is in computing limits of ramification 
points of families of linear systems. The theorem is 
perfectly adapted for dealing with the case of plane 
curves, the study of which will be done in 
\cite{EM}. Example \ref{simple} is given to show, 
in a very simple situation, how the theorem will be 
applied there.

A rough layout of the paper is as follows. In 
Section 2 we define modifications, and present the 
main technical lemmas that will allow us to keep 
track of them later on. Section 3 is devoted to defining 
cycles and limit cycles, and proving a few of their 
fundamental properties, among them Proposition \ref{limsum}, 
stating that taking the fundamental class of the limit is 
additive for Cartier divisors. Section 4 is the 
heart of the notes, containing the main result, 
Theorem \ref{thm}, and the auxiliary Proposition~\ref{redcartier}, 
giving conditions for when the theorem may be applied. Finally, in 
Section 5 we present examples to show how 
Theorem~\ref{thm} can be applied.

\section{Modifications}

\begin{subsct}\label{setup}\setcounter{equation}{0}
({\it Setup.}) 
Throughout the paper, 
$S$ will stand for the spectrum of a discrete valuation 
ring, $s$ for its closed point and $\eta$ for its 
generic point. Also, $\pi$ will denote a local parameter of 
$S$ at $s$. 

Throughout the paper, $f\: X\to S$ will stand for 
a map from a Noetherian scheme $X$. Set 
$X_s:=f^{-1}(s)$ and $X_\eta:=f^{-1}(\eta)$. We call 
$X_s$ the \emph{special fiber} and $X_\eta$ the 
\emph{generic fiber} of $f$. We will always assume 
$X_s$ has no embedded components. Denote by 
$C_1,\dots,C_n$ the subschemes of $X_s$ defined by the 
primary ideal sheaves of 0 in $\Ocal_{X_s}$, and 
$\xi_1,\dots,\xi_n$ their generic points. We say 
that $C_1,\dots,C_n$ are the 
\emph{irreducible primary subschemes} of $X_s$.

A union of irreducible primary subschemes of $X_s$, defined 
by the intersection of the corresponding sheaves of ideals, will 
be called a \emph{primary subscheme} of $X_s$. 
If $Y$ is a primary subscheme of $X_s$, the union of all 
the irreducible primary subschemes not contained in $Y$ 
will be called the 
\emph{complementary primary subscheme to $Y$} and denoted 
$Y^c$. By definition, the empty set and 
$X_s$ are to be considered primary subschemes of $X_s$. 

Let $\Div(X)$ denote the group of Cartier divisors of 
$X$, and $\Div^+(X)$ the submonoid of effective 
Cartier divisors. We will view an element of 
$\Div^+(X)$ as a closed subscheme of $X$. Conversely, 
we will write $Y\in\Div^+(X)$ for any closed subscheme 
of $X$ defined locally everywhere by a nonzero divisor.

For each closed subscheme $Y$ of $X$, let 
$\Ical_Y$ denote its sheaf of ideals. If $Z$ is another 
closed subscheme, we write $Y\leq Z$ if $Y\subseteq Z$. If 
$D\in\Div^+(X)$, let $Y+D$ denote the 
closed subscheme of $X$ given by the sheaf of ideals 
$\Ical_D\Ical_Y$. In addition, if $D\subseteq Y$, 
let $Y-D$ denote the residual subscheme, given 
by the conductor ideal $(\Ical_Y:\Ical_D)$. Of course, 
$Y-D\leq Y\leq Y+D$ and $Y=(Y-D)+D=(Y+D)-D$.

Let $\Twist(f)$ denote the free Abelian group generated 
by $C_1,\dots,C_n$. An element of 
$\Twist(f)$ will be called 
a \emph{twister}. We say that a twister
$\gamma=\sum_im_iC_i$ is \emph{effective} if $m_i\geq 0$ 
for each $i=1,\dots,n$, and \emph{reduced} 
if, in addition, 
$m_i\leq 1$ for each $i=1,\dots,n$. 
Let $\Twist^+(f)\subset\Twist(f)$ denote the submonoid of 
effective twisters. 
We can naturally identify the set of 
primary subschemes of $X_s$ with the set of reduced 
effective twisters.
\end{subsct}

\begin{subsct}\setcounter{equation}{0}
({\it Modifications by primary subschemes.}) 
Let $\Jcal$ be a coherent sheaf on $X$, 
and $Y$ a primary subscheme of $X_s$. Let 
$\Jcal_Y$ denote the restriction of 
$\Jcal$ to $Y$ modulo torsion. 
In other words, $\Jcal_{Y}$ is the 
image of the natural map 
      $$\Jcal|_{Y}\lra
        \bigoplus_{\xi_i\in Y}(\Jcal|_Y)_{\xi_i}$$
Let $\Jcal(-Y)$ denote the kernel of 
the quotient map $\Jcal\to\Jcal_Y$. We say that 
$\Jcal(-Y)$ is a \emph{modification by $Y$ of} $\Jcal$. 
By definition, $\Jcal_\emptyset=0$ and 
$\Jcal(-\emptyset)=\Jcal$.

(We will never use the above construction for a sheaf denoted $\Ical$. 
So, throughout the paper, $\Ical_Y$ will always be understood as the 
sheaf of ideals of a subscheme $Y$ of $X$.)

Notice that $\Jcal(-Y)$ is also the kernel of the natural map
      $$\Jcal\lra\bigoplus_{\xi_i\in Y}
        (\Jcal|_{X_s})_{\xi_i}.$$ 
Thus, for each primary subscheme $Z$ of $X_s$ containing 
$Y$, we have that $\Jcal(-Z)\subseteq\Jcal(-Y)$. 
In addition, $\pi\Jcal\subseteq\Jcal(-X_s)$. Hence, 
there is a natural map, $\Jcal\to\Jcal(-Y)$, obtained 
as the composition,
      $$\Jcal\lra\pi\Jcal\lra\Jcal(-X_s)\lra\Jcal(-Y),$$
where the first map is multiplication by $\pi$.

If $\Lcal$ is an invertible sheaf on $X$, then
$(\Jcal\ox\Lcal)(-Y)=\Jcal(-Y)\ox\Lcal$, as 
subsheaves of $\Jcal\ox\Lcal$.

The sheaves $\Jcal_Y$ and $\Jcal(-Y)$ and all of the maps above are functorial on 
$\Jcal$. So are the maps and the statements of the proposition below.
\end{subsct}

\begin{proposition}\setcounter{equation}{0}\label{prep}
Let $\Jcal$ be a 
coherent sheaf on $X$. Then the following three 
statements hold.
\begin{enumerate}
\item For all primary subschemes $Y$ and $Z$ of 
$X_s$,
      $$\Jcal(-Y)(-Z)=\Jcal(-Z)(-Y).$$
\item For all primary subschemes $Y$ and $Z$ of 
$X_s$ such that $Z\subseteq Y^c$, the inclusions 
$\Jcal(-Y)\to\Jcal$ and $\Jcal(-Z)\to\Jcal$ induce 
injections $\Jcal(-Y)_Z\to\Jcal_Z$ and 
$\Jcal(-Z)_Y\to\Jcal_Y$ whose cokernels are isomorphic.
\item For all primary subschemes $Y_1$, $Y_2$ and 
$Y_3$ of $X_s$ with $Y_2\subseteq Y_1^c$ and 
$Y_3\subseteq Y_2^c$, the inclusion 
$\Jcal(-Y_1\cup Y_2)\to\Jcal(-Y_1)$ induces a short exact 
sequence:
      $$0\to\Jcal(-Y_1\cup Y_2)_{Y_3}
         \lra\Jcal(-Y_1)_{Y_2\cup Y_3}
	 \lra\Jcal(-Y_1)_{Y_2}\to 0.$$
\end{enumerate}
\end{proposition}

\begin{proof} Clearly, $\Jcal(-Y)|_{X-Y}=\Jcal|_{X-Y}$. 
In particular, the natural map 
      $$(\Jcal(-Y)|_{X_s})_{\xi_i}\longrightarrow
        (\Jcal|_{X_s})_{\xi_i}$$
is bijective for each $\xi_i\not\in Y$. Therefore, 
$\Jcal(-Y)(-Z)=\Jcal(-Y\cup Z)$ 
if $Z\subseteq Y^c$. More generally, 
writing $Y=Y'\cup W$ and $Z=Z'\cup W$, where 
$Y'$, $Z'$ and $W$ are primary subschemes such that 
$Y'\subseteq Z^c$ and $Z'\subseteq Y^c$, we have
     \begin{align*}
     \Jcal(-Y)(-Z)=&\Jcal(-Y')(-W)(-Z')(-W)=
     \Jcal(-Y')(-Z)(-W)\\
     =&\Jcal(-Y'\cup Z)(-W)=\Jcal(-Z)(-Y')(-W)\\
     =&\Jcal(-Z)(-Y).
     \end{align*}

As for the second statement, since 
$\Jcal(-Y)_{\xi_i}=\Jcal_{\xi_i}$ for every 
$\xi_i\in Z$, it follows that the naturally induced 
map $\Jcal(-Y)_Z\to\Jcal_Z$ is injective. An analogous 
reasoning shows that $\Jcal(-Z)_Y\to\Jcal_Y$ is also 
injective. Now,
      $$\frac{\Jcal_Z}{\Jcal(-Y)_Z}=
        \frac{\Jcal/\Jcal(-Z)}{\Jcal(-Y)/\Jcal(-Y)(-Z)}
	=\frac{\Jcal}{\Jcal(-Y)+\Jcal(-Z)}.$$
By symmetry, $\Jcal_Y/\Jcal(-Z)_Y$ is thus 
isomorphic to $\Jcal_Z/\Jcal(-Y)_Z$.

As for the third statement, consider 
the natural short exact sequence:
      $$0\to\frac{\Jcal(-Y_1\cup Y_2)}
        {\Jcal(-Y_1\cup Y_2)(-Y_3)}\to
	\frac{\Jcal(-Y_1)}{\Jcal(-Y_1\cup Y_2)(-Y_3)}\to
	\frac{\Jcal(-Y_1)}{\Jcal(-Y_1\cup Y_2)}\to 0.$$
By definition, the first quotient is 
$\Jcal(-Y_1\cup Y_2)_{Y_3}$, while the last is 
$\Jcal(-Y_1)_{Y_2}$. Now, using the first statement, 
      $$\Jcal(-Y_1\cup Y_2)(-Y_3)=\Jcal(-Y_1)(-Y_2)(-Y_3)=
        \Jcal(-Y_1)(-Y_2\cup Y_3).$$
So, we may identify the middle quotient with 
$\Jcal(-Y_1)_{Y_2\cup Y_3}$, and thus obtain the desired 
short exact sequence.
\end{proof}

\begin{subsct}\setcounter{equation}{0}
({\it Modification by twisters.}) Let $\Jcal$ be a 
coherent sheaf on $X$. For each $\gamma\in\Twist^+(f)$ define
a subsheaf $\Jcal^{\gamma}$ of $\Jcal$ 
recursively as follows: if 
$\gamma=0$, then $\Jcal^\gamma:=\Jcal$; if 
$\gamma\neq 0$, then let 
     $$\Jcal^{\gamma}:=\Jcal^{\gamma-C_i}(-C_i)$$ 
for any $C_i$ such that $\gamma-C_i$ is effective. 
It follows from the first statement of 
Proposition \ref{prep} that 
$\Jcal^\gamma$ is well-defined, and 
     $$\Jcal^{\gamma_1+\gamma_2}=
       (\Jcal^{\gamma_1})^{\gamma_2}$$
for every two $\gamma_1,\gamma_2\in\Twist^+(f)$.
We call $\Jcal^\gamma$ the 
\emph{$\gamma$-modification of $\Jcal$}.

If $\Lcal$ is an invertible sheaf on $X$, 
then $\Jcal^\gamma\ox\Lcal=(\Jcal\ox\Lcal)^\gamma$ as 
subsheaves of $\Jcal\ox\Lcal$. Also, the $\gamma$-modifications $\Jcal^\gamma$ 
are functorial on $\Jcal$.
\end{subsct}

\begin{subsct}\setcounter{equation}{0}
({\it Torsion-free rank-1 sheaves.}) 
Let $\Jcal$ be an $S$-flat coherent sheaf on $X$. 
We say that $\Jcal$ is \emph{torsion-free on} $X/S$ 
if the associated components 
of $\Jcal|_{X_s}$ are components of $X_s$, or equivalently, 
if the natural map 
$\Jcal|_{X_s}\to\Jcal_{X_s}$ is a bijection. 
We say that $\Jcal$ is \emph{of rank $1$ on} $X/S$ if 
$(\Jcal|_{X_s})_{\xi_i}\cong\Ocal_{X_s,\xi_i}$ 
for each $i=1,\dots,n$.
\end{subsct}

\begin{proposition}\setcounter{equation}{0}\label{prop} 
Let $\Jcal$ be a torsion-free sheaf on $X/S$, and 
$Y$ a primary subscheme of $X_s$. Set $Z:=Y^c$.
Then the following four statements hold.
\begin{enumerate}
\item $\pi\Jcal=\Jcal(-X_s)$ and the natural maps 
       $$\Jcal|_{X_s}\lra\Jcal_{X_s}\quad\text{and}\quad 
         \Jcal\lra\pi\Jcal$$
are isomorphisms.
\item $\Jcal(-Y)$ is torsion-free on $X/S$.
\item  The natural maps 
$\Jcal(-Y)\to\Jcal$ and $\Jcal\to\Jcal(-Y)$ are injective 
and induce short exact sequences:
      \begin{align*}
      0&\lra\Jcal(-Y)_Z\lra\Jcal|_{X_s}\lra\Jcal_Y
      \lra 0,\\
      0&\lra\Jcal_Y\lra\Jcal(-Y)|_{X_s}\lra\Jcal(-Y)_Z
      \lra 0.
      \end{align*}
\item If $\Jcal$ is of rank 1, so is $\Jcal(-Y)$.
\end{enumerate}
\end{proposition}

\begin{proof} The first map, 
$\Jcal|_{X_s}\to\Jcal_{X_s}$, is an isomorphism 
because the associated components of 
$\Jcal|_{X_s}$ are components of 
$X_s$. Clearly, it follows that $\Jcal(-X_s)=\pi\Jcal$. 
In addition, since $\Jcal$ is $S$-flat, the 
multiplication-by-$\pi$ 
map $\Jcal\to\pi\Jcal$ is an isomorphism.

As for the second statement, 
since $S$ is the spectrum of a discrete valuation ring, 
and $\Jcal$ is $S$-flat, also its subsheaf $\Jcal(-Y)$ is $S$-flat. 
In addition, since the multiplication-by-$\pi$ 
bijection $\Jcal\to\pi\Jcal$ carries $\Jcal(-Y)$ onto 
$(\pi\Jcal)(-Y)$, by functoriality, we have
      $$\Jcal(-Y)(-X_s)=\Jcal(-X_s)(-Y)=
      (\pi\Jcal)(-Y)=\pi\Jcal(-Y),$$
and thus the natural map 
$\Jcal(-Y)|_{X_s}\to\Jcal(-Y)_{X_s}$ is an isomorphism. 
So $\Jcal(-Y)$ is torsion-free.

Consider now the third statement. First, recall that the 
natural map $\Jcal(-Y)\to\Jcal$ is an inclusion, whence 
injective. Second, the natural map $\Jcal\to\Jcal(-Y)$ 
is injective if and only if the natural map 
$\Jcal\to\pi\Jcal$ is an isomorphism, 
and this is the case by the first statement.

As for the exact sequences, the first is obtained from 
that in Proposition \ref{prep} by setting 
$Y_1:=\emptyset$, $Y_2:=Y$ and $Y_3:=Z$, and recalling 
from the first statement that $\Jcal|_{X_s}=\Jcal_{X_s}$. 

The second is also obtained from that in 
Proposition \ref{prep}, this time by setting 
$Y_1:=Y$, $Y_2:=Z$ and $Y_3:=Y$. However, we use 
the composition of isomorphisms,
      $$\Jcal\lra\pi\Jcal\lra\Jcal(-X_s),$$ 
to replace the leftmost sheaf $\Jcal(-X_s)_Y$ 
with $\Jcal_Y$, and we use 
that $\Jcal(-Y)$ is torsion-free, 
to replace $\Jcal(-Y)_{X_s}$ with $\Jcal(-Y)|_{X_s}$.

The fourth statement follows from the two exact 
sequences of the third statement. Indeed, the first one 
yields 
$(\Jcal(-Y)|_{X_s})_{\xi_i}\cong(\Jcal|_{X_s})_{\xi_i}$ 
for each $\xi_i\in Z$, while the second one 
yields 
$(\Jcal|_{X_s})_{\xi_i}\cong(\Jcal(-Y)|_{X_s})_{\xi_i}$ 
for each $\xi_i\in Y$. 
Thus $\Jcal(-Y)$ is of rank 1 if and only 
if so is $\Jcal$.
\end{proof}

\section{Limits of Cartier divisors}

\begin{subsct}\setcounter{equation}{0} 
({\it Cycles.}) Assume $X_s$ is of pure dimension, 
say $d$. Let $\Cyc(X_s)$ denote the free Abelian 
group generated by all integral closed 
subschemes of $X_s$ of dimension $d-1$. We will simply 
say that an element of $\Cyc(X_s)$ is a \emph{cycle}. 
A cycle is called \emph{effective} if its expression 
as a $\text{\bf Z}$-linear combination of 
integral subschemes involves only nonnegative 
coefficients. Let $\Cyc^+(X_s)\subset\Cyc(X_s)$ 
denote the submonoid of effective cycles.

For any coherent sheaf $\Fcal$ on $X_s$ with support 
of dimension at most $d-1$, let 
      $$[\Fcal]:=\sum_Y\ell(\Fcal_{\xi_Y})[Y]
                \in\Cyc^+(X_s),$$
where the sum runs over all irreducible components 
$Y$ of dimension $d-1$ of the support of $\Fcal$, with 
$\xi_Y$ denoting the generic point of $Y$. Since 
localization is exact and length is additive, if 
      $$0\to\Fcal\to\Gcal\to\Hcal\to 0$$
is a short exact sequence of coherent sheaves on $X_s$ 
with support of dimension at most $d-1$, then 
$[\Fcal]=[\Gcal]+[\Hcal]$, a fact we refer to as the 
``additiveness of the bracket.'' Notice as well that, 
if $\Lcal$ is an invertible sheaf on $X_s$, then 
$[\Fcal\ox\Lcal]=[\Fcal]$.

If $W\subset X_s$ is a closed subscheme of dimension 
at most $d-1$, let $[W]:=[\Ocal_W]$. We call 
$[W]$ the \emph{fundamental class} of $W$.
\end{subsct}

\begin{lemma}\setcounter{equation}{0}\label{prip} Assume 
$X_s$ is of pure dimension.
Let $\Fcal$ be a coherent sheaf on $X$ and $\Gcal\subseteq\Fcal$ 
a coherent subsheaf. Let $C_{i_1},\dots,C_{i_m}$ be a 
collection of distinct irreducible primary subschemes of $X_s$. Suppose 
$\Fcal_{\xi_{i_j}}=\Gcal_{\xi_{i_j}}$ for each $j=1,\dots,m$. Set 
$$
Z_j:=C_{i_1}\cup\cdots\cup C_{i_j}\quad\text{and}\quad
Z'_j:=C_{i_{j+1}}\cup\cdots\cup C_{i_m}
$$
for each $j=0,\dots,m$. Then 
\begin{equation}\label{decform}
\Bigg[\frac{\Fcal_{Z_m}}{\Gcal_{Z_m}}\Big]=\sum_{j=0}^{m-1}\Bigg[
\frac{\Fcal(-Z_j)_{C_{i_{j+1}}}}{\Gcal(-Z_j)_{C_{i_{j+1}}}}\Bigg].
\end{equation}
\end{lemma}

\begin{proof} For each $j=0,\dots,m-1$, apply the third statement of 
Proposition~\ref{prep} with $Y_1:=Z_j$, $Y_2:=C_{i_{j+1}}$ 
and $Y_3:=Z'_{j+1}$ to both $\Jcal:=\Fcal$ and $\Jcal:=\Gcal$. By functoriality, 
we get a natural map of short exact sequences:
       $$\begin{CD}
         0 \to \Gcal(-Z_{j+1})_{Z'_{j+1}} @>>>
         \Gcal(-Z_j)_{Z'_j} @>>> \Gcal(-Z_j)_{C_{i_{j+1}}}
	 \to 0\\
         @VVV @VVV @VVV\\
         0 \to \Fcal(-Z_{j+1})_{Z'_{j+1}} @>>>
         \Fcal(-Z_j)_{Z'_j} @>>> \Fcal(-Z_j)_{C_{i_{j+1}}}
	 \to 0.
          \end{CD}$$
Always, the vertical map to the right is injective with cokernel 
supported in codimension 1 in $C_{i_{j+1}}$, because 
$\Gcal_{\xi_{i_{j+1}}}=\Fcal_{\xi_{i_{j+1}}}$. Thus, all the 
vertical maps are injective with cokernel supported in codimension 1 
in $X_s$, and \eqref{decform} holds by the snake lemma.     
\end{proof}

\begin{subsct}\setcounter{equation}{0} 
({\it Limits of Cartier divisors.}) Assume 
that $f\:X\to S$ is flat, or equivalently, that 
$X_s$ is a Cartier divisor of $X$. Assume as well that $X_s$ is of 
pure dimension. Since $X_s$ is a Cartier divisor of $X$, also $X$, 
and thus $X_\eta$, are of pure dimension. 
For each closed subscheme $D$ of $X$, let 
       $$\lim D:=X_s\cap\ol{D\cap X_\eta}^X.$$
We call $\lim D$ the \emph{limit subscheme of $D$}.

Suppose $D\cap X_\eta$ is a Cartier divisor. 
Since $X_\eta$ is of pure dimension, $D\cap X_\eta$ 
is of pure codimension 1 in $X_\eta$. Thus, since 
$\ol{D\cap X_\eta}^X$ is $S$-flat, also 
$\lim D$ is of pure codimension 1 in $X_s$. 
Let $[\lim D]$ 
denote the associated cycle. We call $[\lim D]$ the 
\emph{limit cycle of $D$}. 

The operation 
$D\mapsto[\lim D]$ induces a homomorphism  
of monoids from $\Div^+(X)$ to $\Cyc^+(X_s)$, as a 
consequence of the proposition below. 
\end{subsct}

\begin{proposition}\setcounter{equation}{0}\label{limsum}
Assume that $f\:X\to S$ is flat and $X_s$ has pure 
dimension. Let $D_1$, $D_2$ and $D_3$ be $S$-flat 
closed subschemes of $X$ of pure codimension $1$. Assume 
that $D_1\cap X_\eta$ is a Cartier divisor of $X_\eta$ and 
       $$D_3\cap X_\eta=(D_1\cap X_\eta)+
         (D_2\cap X_\eta).$$
Then
       $$[D_3\cap X_s]=[D_1\cap X_s]+[D_2\cap X_s].$$  
\end{proposition}

(This proposition is a slight generalization of \cite{No}, 
Prop.~5.12, p.~49.)

\begin{proof} For each $i=1,2,3$, 
since $D_i$ is $S$-flat of 
pure codimension 1, also $D_i\cap X_s$ is 
of pure codimension 1 in $X_s$. Again by 
flatness, $D_i$ is the closure of $D_i\cap X_\eta$. 
Thus, from the hypotheses, 
we get that $D_3=D_1\cup D_2$ set-theoretically, and hence
       $$D_3\cap X_s=(D_1\cap X_s)\cup(D_2\cap X_s)$$
set-theoretically.

Let $W\subseteq D_3\cap X_s$ be an irreducible component. 
We need only show that the coefficient of $[W]$ in 
the expression for $[D_3\cap X_s]$ is the sum of those 
for $[D_1\cap X_s]$ and $[D_2\cap X_s]$. 
Let $\zeta\in W$ be 
the generic point, and set $A:=\Ocal_{X,\zeta}$. Let 
$I_1$, $I_2$ and $I_3$ be the respective ideals of $D_1$, 
$D_2$ and $D_3$ in $A$. 

Let $Y_1,\dots,Y_r$ be the 
irreducible components of $D_3$ containing $W$. These 
correspond to the minimal prime ideals 
$\mathfrak p_1,\dots,\mathfrak p_r$ of $A$ containing 
$I_3$. Notice that, for $i=1,2$, since $D_i$ and $D_3$ 
have the same pure dimension, and $D_3\supseteq D_i$, 
the minimal prime ideals of $A$ 
containing $I_i$ are those $\mathfrak p_j$ 
such that $\mathfrak p_j\supseteq I_i$. 

Since $D_i$ is $S$-flat, $\pi$ is a nonzero-divisor of $A/I_i$ 
for each $i=1,2,3$. In particular, 
$\pi{\not\in}\mathfrak p_j$ for any $j$. 
By \cite{G}, Lemme 21.10.17.7, p.~299 or 
\cite{F}, Lemma A.2.7, p.~410, for $i=1,2,3$,
        \begin{equation}\label{ell}
	  \ell(A/(I_i+\pi A))=\sum_{j=1}^r
          \ell(A_{\mathfrak p_j}/I_iA_{\mathfrak p_j})
	  \ell(A/(\mathfrak p_j+\pi A)).
	\end{equation}

The left-hand side of \eqref{ell} 
is the coefficient of $[W]$ in the 
expression for $[D_i\cap X_s]$. Thus, we need only show 
that, for each $j=1,\dots,r$,
        \begin{equation}\label{li}
	  \ell(A_{\mathfrak p_j}/I_3A_{\mathfrak p_j})=
	  \ell(A_{\mathfrak p_j}/I_1A_{\mathfrak p_j})+
          \ell(A_{\mathfrak p_j}/I_2A_{\mathfrak p_j}).
        \end{equation}
Now, since $\pi\not\in{\mathfrak p_j}$, we have 
$I_iA_{\mathfrak p_j}=I_iA_{\pi}A_{\mathfrak p_j}$ for 
$i=1,2,3$. By the hypotheses of the proposition, 
$I_3A_{\pi}=I_1I_2A_{\pi}$, and 
there is a nonzero-divisor $f_j\in A_{\mathfrak p_j}$ such that 
$I_1A_{\mathfrak p_j}=f_jA_{\mathfrak p_j}$. Since $f_j$ is not a 
zero-divisor, multiplication 
by $f_j$ induces a short exact sequence:
        $$0\to\frac{A_{\mathfrak p_j}}{I_2A_{\mathfrak p_j}} \lra
	  \frac{A_{\mathfrak p_j}}{f_jI_2A_{\mathfrak p_j}} \lra
          \frac{A_{\mathfrak p_j}}{f_jA_{\mathfrak p_j}}\to 0.$$
The additiveness 
of the length yields \eqref{li}.
\end{proof}

\section{The main theorem}

\begin{theorem}\setcounter{equation}{0}\label{thm}
Assume that $f\:X\to S$ is flat and $X_s$ has pure dimension. 
Let $D$ be an effective Cartier divisor of $X$. 
Suppose that, for each $i=1,\dots,n$, 
there are effective Cartier divisors $E_i$ and $F_i$ of $X$ 
and a nonnegative integer $p_i$ such that 
$\xi_i\not\in E_i+F_i$ and $D+E_i=p_iX_s+F_i$. Then
      $$[\lim D]=\sum_{i=1}^n\Big([F_i\cap C_i]
        -[E_i\cap C_i]\Big).$$
\end{theorem}

\begin{proof} Let $\gamma=\sum_ir_iC_i$ 
be an effective twister for which 
$\Ical_D\subseteq\Ocal_X^\gamma$. Suppose 
$r:=r_1+\cdots+r_n$ is maximal for this property. (Such $\gamma$ exists 
because $\ell(\Ocal_{D,\xi_i})<\infty$ for each $i=1,\dots,n$.)
Then, since $X$ is $S$-flat,
      $$\Ical_D(-\textstyle\sum_i\sum_{j\neq i} r_iC_j)
        \subseteq\Ocal_X(-r(C_1+\cdots+C_n))=
	\pi^r\Ocal_X.$$
Let 
      $$\Jcal:=\Bigg(
        \Ical_D\Big(-\sum_{i=1}^n\sum_{j\neq i} r_iC_j\Big):
	\pi^r\Ocal_X\Bigg)\subseteq\Ocal_X.$$
Then $\Jcal^\gamma=\Ical_D$. Note that, 
since $\Ical_D$ is invertible, and hence 
torsion-free of rank 1 on $X/S$, also $\Jcal$ is 
torsion-free of rank 1 by Proposition~\ref{prop}. In 
addition, $\Jcal\not\subseteq\Ocal_X(-C_i)$ for every 
$i=1,\dots,n$, by the maximality of $r$. After 
reordering the components $C_i$, we may assume that
      \begin{equation}\label{orderr}
	r_1\geq r_2\geq\cdots\geq r_{n-1}\geq r_n.
      \end{equation}

If $\Kcal$ is a sheaf of ideals and $G$ is a 
Cartier divisor of $X$, then 
the multiplication map $\Kcal\ox\Ical_G\to\Kcal\Ical_G$ 
is an isomorphism. Thus, for each $\mu\in\Twist^+(f)$, 
since 
$(\Kcal\ox\Ical_G)^\mu=\Kcal^\mu\ox\Ical_G$ 
as subsheaves of $\Kcal\ox\Ical_G$, and since the multiplication 
map carries $\Kcal^\mu\ox\Ical_G$ onto 
$\Kcal^\mu\Ical_G$ and $(\Kcal\ox\Ical_G)^\mu$ 
onto $(\Kcal\Ical_G)^\mu$, we have
      \begin{equation}\label{IIG}
        \Kcal^\mu\Ical_G=(\Kcal\Ical_G)^\mu.
      \end{equation}
In addition, if $G\cap Y$ is Cartier for a certain 
primary subscheme $Y$, then $\Ical_G|_Y=\Ical_{G\cap Y|Y}$, 
and it follows that
      \begin{equation}\label{IIGY}
	(\Kcal\Ical_G)_Y=\Kcal_Y\Ical_{G\cap Y|Y}.
      \end{equation}

By the hypothesis of the theorem, for each $i=1,\dots,n$,
      $$\pi^{p_i}\Ical_{F_i}=\Ical_{p_iX_s+F_i}=
        \Ical_{D+E_i}=\Ical_D\Ical_{E_i}=
	\Jcal^{\gamma}\Ical_{E_i}.$$
Using Equation \eqref{IIG}, we get that
$\pi^{p_i}\Ical_{F_i}=(\Jcal\Ical_{E_i})^\gamma$.
Now, since $\xi_i\not\in E_i$ and $\Jcal\not\subseteq\Ocal_X(-C_i)$, we get
that $r_i$ is the largest integer $j$ such that 
$(\Jcal\Ical_{E_i})^\gamma\subseteq\Ocal_X(-jC_i)$. 
On the other hand, since also $\xi_i\not\in F_i$, 
we get that $p_i$ is the largest integer 
$j$ such that 
$\pi^{p_i}\Ical_{F_i}\subseteq\Ocal_X(-jC_i)$. Since 
$\pi^{p_i}\Ical_{F_i}=(\Jcal\Ical_{E_i})^\gamma$, we 
have $p_i=r_i$. Putting
       $$\alpha_i:=\sum_{j<i}(r_j-r_i)C_j\quad
         \text{and}\quad
	 \beta_i:=\sum_{j>i}(r_i-r_j)C_j,$$
we get
      \begin{equation}\label{JEFBA}
	(\Jcal\Ical_{E_i})^{\alpha_i}=
        \Ical_{F_i}^{\beta_i}.
      \end{equation}

Let 
      \begin{equation}\label{gamma'}
	\gamma':=\sum_jr'_jC_j,\quad\text{where }r'_j:=r_1-r_j\text{ for each }
	j=1,\dots,n,
      \end{equation}
and set 
      \begin{equation}\label{delta}
	\delta_i:=\sum_{j<i}r'_jC_j+\sum_{j\geq i}r'_iC_j,\\
	\quad\text{and}\quad
        \epsilon_i:=\sum_{j< i}r'_iC_j+\sum_{j\geq i}r'_jC_j.
      \end{equation}
Then $\alpha_i+\delta_i=r'_i(C_1+\cdots+C_n)$ and $\beta_i+\delta_i=\gamma'$. 
Since $\Jcal\Ical_{E_i}$ is torsion-free of rank 1, 
it follows from \eqref{JEFBA} that
        \begin{equation}\label{JEF}
	  \begin{aligned}
	    \pi^{r'_i}\Jcal\Ical_{E_i}=&
            (\Jcal\Ical_{E_i})^{\alpha_i+\delta_i}=
            ((\Jcal\Ical_{E_i})^{\alpha_i})^{\delta_i}\\
	    =&(\Ical_{F_i}^{\beta_i})^{\delta_i}=
            \Ical_{F_i}^{\beta_i+\delta_i}
	    =\Ical_{F_i}^{\gamma'}.
	  \end{aligned}
	\end{equation}
Notice, for later use, that $\epsilon_i=\alpha_i+\gamma'$.

Since $\xi_i\not\in E_i+F_i$, and since
$C_i$ is not a summand of $\alpha_i$ or 
$\beta_i$, it follows from \eqref{JEFBA} that
      $$\Jcal_{\xi_i}=(\Jcal\Ical_{E_i})_{\xi_i}=
        ((\Jcal\Ical_{E_i})^{\alpha_i})_{\xi_i}
        =(\Ical_{F_i}^{\beta_i})_{\xi_i}
	=(\Ical_{F_i})_{\xi_i}
	=\Ocal_{X,\xi_i}.$$
Since this holds for each $i=1,\dots,n$, and since 
$\Jcal$ is torsion-free on $X/S$, it 
follows that the induced 
map $\Jcal|_{X_s}\to\Ocal_{X_s}$ is injective. So the 
inclusion $\Jcal\to\Ocal_X$ has flat cokernel. 
Since $\Jcal^\gamma=\Ical_D$, we have 
that $\Jcal|_{X_\eta}$ is the sheaf of ideals of  
$D\cap X_\eta$ in $X_\eta$. Thus $\lim D$ is the subscheme of $X_s$ 
with ideal sheaf $\Jcal|_{X_s}$, and hence
       \begin{equation}\label{limD}
	 [\lim D]=\bigg[\frac{\Ocal_{X_s}}
         {\Jcal|_{X_s}}\bigg].
       \end{equation}

Set $Z_1:=\emptyset$ and, for each $j=2,\dots,n+1$, put 
$Z_j:=C_1\cup\cdots\cup C_{j-1}$. By 
Lemma~\ref{prip}, 
        $$\bigg[\frac{\Ocal_{X_s}}{\Jcal|_{X_s}}\bigg]=
          \sum_{i=1}^n\bigg[\frac{\Ocal_X(-Z_i)_{C_i}}
          {\Jcal(-Z_i)_{C_i}}\bigg],$$
and thus, by \eqref{limD} and the additiveness of the bracket,
      \begin{equation}\label{OJ1}
	[\lim D]=\sum_{i=1}^n\Bigg(
	\bigg[\frac{\Ocal_{C_i}}{\Jcal_{C_i}}\bigg]-
	\bigg[\frac{\Ocal_{C_i}}{\Ocal_X(-Z_i)_{C_i}}\bigg]
	+\bigg[\frac{\Jcal_{C_i}}{\Jcal(-Z_i)_{C_i}}\bigg]
	\Bigg).
      \end{equation}

      Now, using \eqref{IIG}, \eqref{IIGY} and 
\eqref{JEF}, and using that both $C_i\cap E_i$ and $C_i\cap F_i$ are Cartier, 
we get
      \begin{align*}
	\bigg[\frac{\Jcal_{C_i}}
        {\Jcal(-Z_i)_{C_i}}\bigg]=&
        \bigg[\frac{\Jcal_{C_i}\Ical_{E_i\cap C_i|C_i}}
        {\Jcal(-Z_i)_{C_i}\Ical_{E_i\cap C_i|C_i}}\bigg]
	=\bigg[\frac{(\Jcal\Ical_{E_i})_{C_i}}
	{(\Jcal\Ical_{E_i})(-Z_i)_{C_i}}\bigg]\\
	=&\bigg[\frac{(\pi^{r'_i}\Jcal\Ical_{E_i})_{C_i}}
	{(\pi^{r'_i}\Jcal\Ical_{E_i})(-Z_i)_{C_i}}\bigg]
        =\bigg[\frac{(\Ical_{F_i}^{\gamma'})_{C_i}}
	{\Ical_{F_i}^{\gamma'}(-Z_i)_{C_i}}\bigg]\\
	=&\bigg[\frac{(\Ocal_X^{\gamma'}\Ical_{F_i})_{C_i}}
	{(\Ocal_X^{\gamma'}\Ical_{F_i})(-Z_i)_{C_i}}
	\bigg]
	=\bigg[\frac{(\Ocal_X^{\gamma'})_{C_i}\Ical_{F_i\cap C_i|C_i}}
        {\Ocal_X^{\gamma'}(-Z_i)_{C_i}\Ical_{F_i\cap C_i|C_i}}\bigg]\\
	=&\bigg[\frac{(\Ocal_X^{\gamma'})_{C_i}}
        {\Ocal_X^{\gamma'}(-Z_i)_{C_i}}\bigg]
      \end{align*}
Thus \eqref{OJ1} becomes
      \begin{equation}\label{OJ}
	[\lim D]=\sum_{i=1}^n\Bigg(\bigg[\frac{\Ocal_{C_i}}
	{\Jcal_{C_i}}\bigg]-
        \bigg[\frac{\Ocal_{C_i}}
        {\Ocal_X(-Z_i)_{C_i}}\bigg]+
        \bigg[\frac{(\Ocal_X^{\gamma'})_{C_i}}
        {\Ocal_X^{\gamma'}(-Z_i)_{C_i}}\bigg]\Bigg).
      \end{equation}

Using a similar reasoning, and the additiveness of the bracket,
      \begin{align*}
	\bigg[\frac{\Ocal_{C_i}}{\Jcal_{C_i}}\bigg]=&
	\bigg[\frac{(\Ical_{E_i})_{C_i}}
	{(\Jcal\Ical_{E_i})_{C_i}}\bigg]=
	\bigg[\frac{\Ocal_{C_i}}
	{(\Jcal\Ical_{E_i})_{C_i}}\bigg]-[E_i\cap C_i]\\
	=&\bigg[\frac{(\pi^{r'_i}\Ocal_X)_{C_i}}
	{(\pi^{r'_i}\Jcal\Ical_{E_i})_{C_i}}\bigg]
	-[E_i\cap C_i]
	=\bigg[\frac{(\pi^{r'_i}\Ocal_X)_{C_i}}
	{(\Ical_{F_i}^{\gamma'})_{C_i}}\bigg]
	-[E_i\cap C_i]\\
	=&[F_i\cap C_i]
	+\bigg[\frac{(\pi^{r'_i}\Ical_{F_i})_{C_i}}
	{(\Ical_{F_i}^{\alpha_i+\gamma'})_{C_i}}\bigg]
	-\bigg[\frac{(\Ical_{F_i}^{\gamma'})_{C_i}}
	{(\Ical_{F_i}^{\alpha_i+\gamma'})_{C_i}}\bigg]
	-[E_i\cap C_i]\\
	=&[F_i\cap C_i]-[E_i\cap C_i]
	+\bigg[\frac{(\pi^{r'_i}\Ocal_X)_{C_i}}
	{(\Ocal_X^{\epsilon_i})_{C_i}}\bigg]
	-\bigg[\frac{(\Ocal_X^{\gamma'})_{C_i}}
	{(\Ocal_X^{\epsilon_i})_{C_i}}\bigg],
      \end{align*}
where we used that $\epsilon_i=\alpha_i+\gamma'$. 
Substituting in \eqref{OJ}, 
we see that we need only prove the following claim.

\vskip0.2cm

{\it Claim:} Let $\gamma:=\sum_ir_iC_i$ be an effective 
twister such that \eqref{orderr} holds. Let $\gamma'$ be as in 
\eqref{gamma'}. For each 
$i=1,\dots,n$, let $\epsilon_i\in\Twist(f)$ be as 
in \eqref{delta}, and put
      $$\theta_i(\gamma):=
        \bigg[\frac{(\pi^{r'_i}\Ocal_X)_{C_i}}
	{(\Ocal_X^{\epsilon_i})_{C_i}}\bigg]
        -\bigg[\frac{(\Ocal_X^{\gamma'})_{C_i}}
	{(\Ocal_X^{\epsilon_i})_{C_i}}\bigg]-\bigg[
        \frac{\Ocal_{C_i}}{\Ocal_X(-Z_i)_{C_i}}\bigg]
	+\bigg[\frac{(\Ocal_X^{\gamma'})_{C_i}}
        {\Ocal_X^{\gamma'}(-Z_i)_{C_i}}\bigg].$$
Then $\theta_1(\gamma)+\cdots+\theta_n(\gamma)=0$.

\vskip0.2cm

We will prove the claim by induction on the 
sum $r':=r'_1+\cdots+r'_n$. If $r'=0$, then $\gamma'=0$ and 
$\epsilon_i=0$ for each $i=1,\dots,n$. 
The claim is trivial in this case, as the first and 
second summands of $\theta_i(\gamma)$ are zero, and the third 
and fourth cancel each other, for each $i=1,\dots,n$.

Now, suppose $r'>0$. Then one of the inequalities in 
\eqref{orderr} is strict. Let 
$\ell$ be an integer, between 2 and $n$, such that 
$r_{\ell-1}>r_\ell$. For each $i=1,\dots,n$, let 
$t_i:=r_i$ if $i\neq\ell$ and $t_\ell:=r_\ell+1$. Then 
$t_1\geq\cdots\geq t_n$ as well. Set $\tau:=\sum_it_iC_i$ and 
$\tau':=\sum_it'_iC_i$, where $t'_i:=t_1-t_i$ for each $i=1,\dots,n$. Then 
$t'_i=r'_i$ for every $i\neq\ell$, but $t'_\ell=r'_\ell-1$. So, by 
induction, $\theta_1(\tau)+\cdots+\theta_n(\tau)=0$.

Set
$$
\rho_i:=\sum_{j<i}t'_iC_j+\sum_{j\geq i}t'_jC_j\quad\text{for $i=1,\dots,n$.}
$$
Then $\gamma'=\tau'+C_\ell$ and
      $$\epsilon_i=\begin{cases}
	\rho_i+C_\ell           &\text{if $i<\ell$,}\\
	\rho_\ell+C_1+\cdots+C_\ell&\text{if $i=\ell$,}\\
	\rho_i                  &\text{if $i>\ell$.}
	\end{cases}$$
Using the above formulas and the additiveness 
of the bracket, we get
      $$\theta_i(\gamma)-\theta_i(\tau)=\bigg[
        \frac{\Ocal_X^{\tau'}(-Z_i)_{C_i}}
	{\Ocal_X^{\gamma'}(-Z_i)_{C_i}}\bigg]
        \quad\text{if $i\neq\ell$,}$$
and
      \begin{align*}
	\theta_\ell(\gamma)-\theta_\ell(\tau)=&
	\bigg[\frac{(\pi^{r'_\ell}\Ocal_X)_{C_\ell}}
	{(\Ocal_X^{\epsilon_\ell})_{C_\ell}}\bigg]-
	\bigg[\frac{(\pi^{t'_\ell}\Ocal_X)_{C_\ell}}
	{(\Ocal_X^{\rho_\ell})_{C_\ell}}\bigg]
	-\bigg[\frac{(\Ocal_X^{\gamma'})_{C_\ell}}
	{(\Ocal_X^{\epsilon_\ell})_{C_\ell}}\bigg]\\
	&+\bigg[\frac{(\Ocal_X^{\tau'})_{C_\ell}}
	{(\Ocal_X^{\rho_\ell})_{C_\ell}}\bigg]
	+\bigg[\frac{(\Ocal_X^{\gamma'})_{C_\ell}}
        {\Ocal_X^{\gamma'}(-Z_\ell)_{C_\ell}}\bigg]
	-\bigg[\frac{(\Ocal_X^{\tau'})_{C_\ell}}
        {\Ocal_X^{\tau'}(-Z_\ell)_{C_\ell}}\bigg]\\
	=&\bigg[\frac{(\pi^{r'_\ell}\Ocal_X)_{C_\ell}}
	{(\Ocal_X^{\epsilon_\ell})_{C_\ell}}\bigg]-
	\bigg[\frac{(\pi^{r'_\ell}\Ocal_X)_{C_\ell}}
	{(\pi\Ocal_X^{\rho_\ell})_{C_\ell}}\bigg]
	-\bigg[\frac{(\Ocal_X^{\gamma'})_{C_\ell}}
	{(\Ocal_X^{\epsilon_\ell})_{C_\ell}}\bigg]\\
	&+\bigg[\frac{(\pi\Ocal_X^{\tau'})_{C_\ell}}
	{(\pi\Ocal_X^{\rho_\ell})_{C_\ell}}\bigg]
	+\bigg[\frac{(\Ocal_X^{\gamma'})_{C_\ell}}
        {\Ocal_X^{\gamma'}(-Z_\ell)_{C_\ell}}\bigg]
	-\bigg[\frac{(\Ocal_X^{\tau'})_{C_\ell}}
        {\Ocal_X^{\tau'}(-Z_\ell)_{C_\ell}}\bigg]\\
	=&-\bigg[
	\frac{\Ocal_X^{\gamma'}(-Z_\ell)_{C_\ell}}
	{(\pi\Ocal_X^{\tau'})_{C_\ell}}\bigg]
	-\bigg[\frac{(\Ocal_X^{\tau'})_{C_\ell}}
        {\Ocal_X^{\tau'}(-Z_\ell)_{C_\ell}}\bigg].
      \end{align*}
Thus, since $\theta_1(\tau)+\cdots+\theta_n(\tau)=0$ by induction, 
we need only show that
        \begin{equation}\label{0=0}
	  \sum_{i\neq\ell}\bigg[
          \frac{\Ocal_X^{\tau'}(-Z_i)_{C_i}}
	  {\Ocal_X^{\gamma'}(-Z_i)_{C_i}}\bigg]=
          \bigg[\frac{\Ocal_X^{\gamma'}(-Z_\ell)_{C_\ell}}
	  {(\pi\Ocal_X^{\tau'})_{C_\ell}}\bigg]
	  +\bigg[\frac{(\Ocal_X^{\tau'})_{C_\ell}}
          {\Ocal_X^{\tau'}(-Z_\ell)_{C_\ell}}\bigg].
	\end{equation}

Now, applying Lemma~\ref{prip} twice, we get that
        \begin{align*}
	  \sum_{i\neq\ell}\bigg[
          \frac{\Ocal_X^{\tau'}(-Z_i)_{C_i}}
	  {\Ocal_X^{\gamma'}(-Z_i)_{C_i}}\bigg]
	  &=\sum_{i<\ell}\bigg[
          \frac{\Ocal_X^{\tau'}(-Z_i)_{C_i}}
	  {\Ocal_X^{\gamma'}(-Z_i)_{C_i}}\bigg]
	  +\sum_{i>\ell}\bigg[
          \frac{\Ocal_X^{\tau'}(-Z_i)_{C_i}}
	  {\Ocal_X^{\gamma'}(-Z_i)_{C_i}}\bigg]\\
	  &=\bigg[\frac{(\Ocal_X^{\tau'})_{Z_\ell}}
	  {(\Ocal_X^{\gamma'})_{Z_\ell}}\bigg]
	  +\bigg[\frac{\Ocal_X^{\tau'}
	  (-Z_{\ell+1})_{Z^c_{\ell+1}}}
	  {\Ocal_X^{\gamma'}(-Z_{\ell+1})_{Z^c_{\ell+1}}}
	  \bigg].
	  \end{align*}
In addition, since $\gamma'=\tau'+C_\ell$, it follows 
from the second statement of Proposition \ref{prep} that
          \begin{align*}
	    \bigg[\frac{(\Ocal_X^{\tau'})_{Z_\ell}}
	    {(\Ocal_X^{\gamma'})_{Z_\ell}}\bigg]
	    =&\bigg[\frac{(\Ocal_X^{\tau'})_{C_\ell}}
	    {\Ocal_X^{\tau'}(-Z_\ell)_{C_\ell}}\bigg],\\
	    \bigg[\frac{\Ocal_X^{\tau'}
	    (-Z_{\ell+1})_{Z^c_{\ell+1}}}
	    {\Ocal_X^{\gamma'}(-Z_{\ell+1})_{Z^c_{\ell+1}}}
	    \bigg]
	    =&\bigg[
	    \frac{\Ocal_X^{\tau'}(-Z_{\ell+1})_{C_\ell}}
	    {(\pi\Ocal_X^{\tau'})_{C_\ell}}\bigg].
	  \end{align*}
Equation \eqref{0=0} follows.
\end{proof}

\begin{proposition}\label{redcartier}
\setcounter{equation}{0}  
Assume that $f\: X\to S$ is 
projective and flat, and that 
$S$ is a $k$-scheme for an infinite 
field $k$. Then the following two statements hold:
\begin{enumerate}
\item For each $i=1,\dots,n$ there is 
$G_i\in\Div^+(X)$ such that $\xi_j\not\in G_i$ for 
$j\neq i$ and $G_i$ coincides with $C_i$ at $\xi_i$.
\item If $X_s$ is reduced, for each closed subscheme 
$D$ of $X$ such that $D\cap X_\eta$ is a Cartier 
divisor there are divisors $E_i\in\Div^+(X)$ with 
$\xi_i\not\in E_i$ and nonnegative integers $p_i$ 
such that $D+E_i\geq p_iX_s$ and 
$\xi_i\not\in (D+E_i)-p_iX_s$ for $i=1,\dots,n$.
\end{enumerate}
\end{proposition}

\begin{proof} Let $R$ be the ring of regular functions 
of $S$. Since $S$ is a $k$-scheme, $R$ is a $k$-algebra. 
Since $f$ is projective, $f$ 
factors through an embedding $\iota\:X\to\IP^m_R$, 
where $\IP^m_R:=\text{Proj}(R[t_0,\dots,t_m])$. 
Let $\Ocal_X(1)$ be the restriction to $X$ of the tautological ample 
sheaf of $\IP^m_R$. Then there is an 
integer $d>0$ such that 
$H^1(\IP^m_R,\Ical_{X|\IP^m_R}(d))=0$ and the $d$-th twist 
$\Ical_{C_i}(d)$ of the sheaf of ideals of $C_i$ in $X$ is 
globally generated for every $i=1,\dots,n$.

Let $\xi_{n+1},\dots,\xi_{n+r}$ denote the associated 
points of $X_\eta$. Then 
$\Ical_{C_i}$ is invertible at $\xi_j$ for 
each $j=1,\dots,n+r$. Indeed, this is clearly so 
if $j\neq i$ because $\xi_j\not\in C_i$. On the other hand, 
$(\Ical_{C_i})_{\xi_i}$ is the ideal of 
$\Ocal_{X,\xi_i}$ generated by $\pi$, which is a 
nonzero-divisor because $f$ is flat, 
whence $\Ical_{C_i}$ is also invertible 
at $\xi_i$.

Since $\Ical_{C_i}(d)$ is globally generated, and 
since $H^1(\IP^m_R,\Ical_{X|\IP^m_R}(d))=0$, for 
each $j=1,\dots,n+r$ there exists a degree-$d$ homogeneous polynomial 
$P_j\in R[t_0,\dots,t_m]$ 
generating $\Ical_{C_i}(d)$ at $\xi_j$. Since 
$k$ is infinite, a general linear combination 
$Q_i:=\sum_jc_jP_j$ with $c_j\in k$ generates 
$\Ical_{C_i}(d)$ at $\xi_j$ for every $j=1,\dots,n+r$. 
Let $G_i\subseteq X$ be the subscheme cut out by 
$Q_i=0$. Since $G_i$ does not vanish on $\xi_j$ 
for any $j=n+1,\dots,n+r$, the subscheme $G_i$ is a 
Cartier divisor. It is indeed the Cartier divisor 
required by the first statement.

As for the second statement, as $X_s$ is reduced, 
for each $j=1,\dots,n$ the point $\xi_j$ lies on the nonsingular locus of 
$X_s$, whence on the nonsingular locus of $X$. So the local ring 
$\Ocal_{X,\xi_j}$ is a discrete valuation ring. 

For each $j=1,\dots,n$, consider the ideal of $D$ at 
$\xi_j$. If it were zero, then $D$ would contain any 
irreducible closed subscheme of $X$ containing $\xi_j$. 
However, among those there is at least one irreducible 
component of $X$, whose generic point lies over $\eta$ 
by flatness. Thus $D$ would contain an irreducible 
component of $X_\eta$, contradicting the hypothesis that 
$D\cap X_\eta$ is a Cartier divisor. So the ideal of 
$D$ at $\xi_j$ is nonzero. Since $\Ocal_{X,\xi_j}$ is a 
discrete valuation ring, this ideal is thus a power 
$p_j$ of the maximal ideal. Let $G_1,\dots,G_n$ be the 
Cartier divisors of $X$ claimed in the first statement. 
Set
       $$E_i:=\sum_{p_j<p_i}(p_i-p_j)G_j.$$
Then $E_i$ is an effective Cartier divisor 
of $X$ not containing $\xi_i$. Also, 
$D+E_i\supseteq p_iX_s$. Set $F_i:=D+E_i-p_iX_s$. Then 
$F_i$ is a subscheme of $X$ that does not contain 
$\xi_i$. 
\end{proof}

\section{Examples}

\begin{example}\setcounter{equation}{0} 
(See \cite{F}, Ex.~11.3.2, p.~203, and 
the references listed there.) 
Let $F$, $A_1$, $A_2$, $G_1$ and $G_2$ 
be forms defining hypersurfaces. Assume $FA_i$ and $G_j$ 
have the same degree, for $i=1,2$ and $j=1,2$. Assume as 
well that 
\begin{equation}\label{hyp1}
	\gcd(FA_2,A_1)=1\quad\text{and}\quad
        \gcd(A_1G_2-A_2G_1,F)=1.
      \end{equation}
Consider the pencils $FA_i+tG_i=0$ for $i=1,2$, and 
their intersection. The hypotheses \eqref{hyp1} imply 
that the intersection is proper for a general $t$. 
Indeed, if the intersection were not proper, then 
there would be a nonnegative integer $m$ and forms $L_0,\dots,L_m$ 
of the same positive degree such that 
the polynomial $L_0+TL_1+\cdots+T^mL_m$ 
divides $FA_1+TG_1$ and $FA_2+TG_2$. But then $L_0$ 
would divide $FA_1$, $FA_2$ and $A_1G_2-A_2G_1$.

Let $W$ denote the limit of the intersection of the 
pencils as $t$ goes to 0. The intersection of the hypersurfaces
$FA_1=0$ and $FA_2=0$ is not proper, and thus does not 
reflect $W$ well. However, $FA_2=0$ cuts a 
Cartier divisor on $A_1=0$. In addition, 
     $$A_1(FA_2+tG_2)=A_2(FA_1+tG_1)+t(A_1G_2-A_2G_1),$$
and $A_1G_2-A_2G_1=0$ cuts a Cartier divisor on $F=0$. 
Thus, by Theorem \ref{thm}, 
     \begin{align*}
       [W]&=[FA_2,A_1=0]+[A_1G_2-A_2G_1,F=0]-
       [A_1,F=0]\\
       &=[A_2,A_1=0]+[A_1G_2-A_2G_1,F=0].
     \end{align*}
\end{example}

\begin{example}\setcounter{equation}{0}  
(See \cite{F}, Ex.~11.3.3, p.~203.)
Consider the families of plane curves 
parameterized by $t$:
     $$x^2y-tz^3=0\quad\text{and}\quad
       (x-t^2y)(y^2-t^2x^2)=0.$$
It is easy to compute the intersection of the 
above curves for general $t$, as the second curve is 
a union of lines. Letting $\beta$ be a primitive 
cubic root of unity, we see 
that the intersection is reduced 
and consists of the nine points:
     $$(t^2:1:t\beta^j),\quad
       (1:t:\beta^j)\quad\text{and}\quad(1:-t:-\beta^j)
       \quad\text{for $j=0,1,2$.}$$
As $t$ goes to 0, the $(t^2:1:t\beta^j)$ 
approach $(0:1:0)$, while the remainder approach the 
six points $(1:0:\pm\beta^j)$.

To compute these limits using Theorem \ref{thm}, 
we first use Proposition~\ref{limsum} to reduce the 
problem to that of computing the limits of the Cartier 
divisors cut on $x^2y-tz^3=0$ by $x-t^2y=0$ and by 
the lines $y\pm tx$. Call $D$ the first limit and 
$D_\pm$ the last two limits.

We will actually use Theorem \ref{thm} to 
compute $2[D]$, and then use Proposition~\ref{limsum} 
to get $[D]$. First, $x^2=0$ cuts a Cartier divisor on 
$y=0$. Also, 
       $$y(x-t^2y)^2\equiv(x^2y-tz^3)+tz^3
         \quad\text{mod }t^2,$$
and $z^3=0$ cuts a Cartier divisor on $x^2=0$. Thus, 
by Theorem \ref{thm},
       $$2[D]=[x^2,y=0]+[z^3,x^2=0]-[y,x^2=0]=6[z,x=0].$$
So $[D]=3[(0:1:0)]$.

As for $D_\pm$, first $y=0$ cuts out a Cartier divisor 
on $x^2=0$. Also,
       $$x^2(y\pm tx)=(x^2y-tz^3)+t(z^3\pm x^3),$$
and $z^3\pm x^3$ cuts a Cartier divisor on $y=0$. 
Thus, by Theorem \ref{thm},
      \begin{align*}
	[D_\pm]=&[y,x^2=0]+[z^3\pm x^3,y=0]-[x^2,y=0]\\
	=&\textstyle\sum_{j=0}^2[(1:0:\mp\beta^j)].
	\end{align*}
\end{example}

\begin{example}\label{simple} 
If a plane curve is smooth, its 
flexes are cut out by the Hessian, 
the determinant of 
the symmetric matrix of second-order partial derivatives 
of the form defining the curve. 
However, if the curve has a linear or a multiple 
component, the Hessian vanishes completely on that 
component. 
What are the possible limits of flexes on the 
curve if the curve has such components?

The above question, considered in \cite{EHb}, p.~151, will also 
be considered in more detail 
in \cite{EM}. Here we will just consider the 
simple case where the curve is reduced with just two 
components, and just one of them is linear, and where 
the curve is deformed in first order 
along a general direction.

Let $k$ be an algebraically closed 
field of characteristic zero. For each $P\in k[[t]][x,y,z]$, 
define the derivation 
$D_P:=\d_y(P)\d_x-\d_x(P)\d_y$, where $\d_x$, $\d_y$ and $\d_z$ 
are the canonical partial $k[[t]]$-derivations of 
$k[[t]][x,y,z]$. Notice that $D_P(P)=0$. 
For short, let $P_x:=\d_x(P)$, $P_y:=\d_y(P)$ and $P_z:=\d_z(P)$. 
Define the Hessian determinant $H(P)$ and 
the Wronskian determinant $W(P)$:
       $$H(P):=\begin{vmatrix}
            P_{x,x} & P_{x,y} & P_{x,z}\\
	    P_{y,x} & P_{y,y} & P_{y,z}\\
	    P_{z,x} & P_{z,y} & P_{z,z}\\
	    \end{vmatrix}\quad\text{and}\quad
         W(P):=\begin{vmatrix}
              x    &   y    &   z   \\
	     D_P(x)  &  D_P(y)  &  D_P(z) \\
	    D_P^2(x) & D_P^2(y) & D_P^2(z)\\
	    \end{vmatrix}.$$
If $P$ is homogeneous of degree $p$, it 
follows from applying the Euler Formula twice that
       \begin{equation}\label{zH}
	 z^3H(P)\equiv (p-1)^2W(P)\quad\text{mod }P.
       \end{equation}

Let 
      $$F:=xG-F_1t\in k[[t]][x,y,z],$$ 
where $G$ and $F_1$ are nonzero forms of degrees 
$d-1$ and $d$, respectively, for an integer $d\geq 3$. 
Assume $G$ is irreducible and $\gcd(xz,G)=1$. 
Then $D_F\equiv xD_G$ mod $(G,t)$, and hence $D_F(G)\equiv 0$ mod $(G,t)$.
So, using the mutilinearity of the determinant, the 
product rule for derivations, and \eqref{zH} for 
$P:=G$, we get 
       \begin{equation}\label{zHP}
	 (d-1)^2W(F)\equiv (xz)^3H(G)\quad\text{mod }(G,t).
       \end{equation}
Since $G$ is irreducible and nonlinear, $H(G)$ 
cuts a divisor on the curve defined by $G$. Let 
$R_G$ denote the 0-cycle of $\mathbb P^2$ associated to this 
divisor. Since $\gcd(G,xz)=1$, it follows from 
\eqref{zHP} that also $W(F)_0$ cuts a divisor 
on the curve given by $G$, where $W(F)_0$ is the constant term 
of $W(F)$.

Since $D_F(F)=0$, using the 
multilinearity of the determinant and the product rule for
derivations, we get 
        \begin{equation}\label{GWt}
	  G^3W(F)\equiv tW'\quad\text{mod }F,
	\end{equation}
where $W'$ is the Wronskian determinant:
        $$W':=\begin{vmatrix}
          F_1 & Gy & Gz \\
	  D_F(F_1) & D_F(Gy) & D_F(Gz)\\
	  D^2_F(F_1) & D^2_F(Gy) & D^2_F(Gz)\\
	  \end{vmatrix}.$$
Now, $D_F\equiv -G\d_y$ mod $(x,t)$, and so $D_F(x)\equiv 0$ mod $(x,t)$. 
Thus, using the 
multilinearity of the determinant and the product rule 
for derivations, we get 
        \begin{equation}\label{WLg}
	  W'\equiv -g^3w\quad\text{mod }(x,t),
	\end{equation}
where $g:=G(0,y,z)$ and $w$ is the Wronksian determinant:
        $$w:=\begin{vmatrix}
              f'   &     g'   &     g''  \\
            f'_y   &   g'_y   &   g''_y  \\
	  f'_{y,y} & g'_{y,y} & g''_{y,y}\\
	  \end{vmatrix},$$
with $f':=F_1(0,y,z)$, $g':=gy$ and 
$g'':=gz$. If $F_1$ is general enough ---
more precisely, if 
$F_1\not\in (x,G)$ --- then $w\neq 0$, and hence also 
$W'\not\equiv 0$ mod $(x,t)$. 

Now, since $f'$, $g'$ and $g''$ have degree $d$, applying 
Euler formula three times, we get
        \begin{equation}\label{dwzh}
	  (d-1)^2dw=z^3h,
	\end{equation}
where $h$ is the ``Hessian'' determinant:
       $$h:=\begin{vmatrix}
         f'_{z,z} & g'_{z,z} & g''_{z,z}\\
         f'_{y,z} & g'_{y,z} & g''_{y,z}\\
         f'_{y,y} & g'_{y,y} & g''_{y,y}\\
         \end{vmatrix}.$$
Let $R_x$ denote 
the 0-cycle of $\mathbb P^2$ of the scheme given by 
$h=x=0$.

For any two coprime forms $P_1$ and $P_2$ of 
$k[x,y,z]$, denote by $[P_1\cdot P_2]$ the 0-cycle 
of $\mathbb P^2$ of the scheme given by $P_1=P_2=0$.

Assume $F_1$ is general. Let $S:=\text{Spec}(k[[t]])$ and 
$X\subseteq\text{\bf P}^2_S$ be the subscheme cut out by 
$F$. Let $f\: X\to S$ denote the structure map. 
Let $E\subseteq X$ be the subscheme cut out by $W(F)$. 
Since $F_1$ is general, the general fiber $X_\eta$ of $f$ is 
smooth, and thus, by \eqref{zH} applied to $P:=F$, the subscheme 
$E$ cuts $X_\eta$ in its divisor of flexes plus 3 times the 
hyperplane section given by $z=0$. 
So, $E$ is a Cartier divisor of $X$. 
And, by Proposition~\ref{limsum}, 
$[\lim E]=3[z\cdot(xG)]+R$, where $R$ is the 0-cycle 
of the schematic boundary of the divisor of flexes 
of $X_\eta$.

Using \eqref{GWt} and Theorem \ref{thm}, 
       $$[\lim E]=[W(F)_0\cdot G]+[W'_0\cdot x]-3[G\cdot x],$$
where $W'_0$ is the constant term of $W'$.
In addition, 
by \eqref{zHP}, \eqref{WLg} and \eqref{dwzh},
       \begin{align*}
	 [W(F)_0\cdot G]=&3[x\cdot G]+3[z\cdot G]+R_G,\\
	 [W'_0\cdot x]=&3[G\cdot x]+3[z\cdot x]+R_x.
       \end{align*}
Thus, 
       $$R=R_G+R_x+3(G\cdot x).$$
\end{example}


\begin{thebibliography}{9}

\bibitem{EH}
D.~Eisenbud and J.~Harris,
\emph{Limit linear series: Basic theory},
Invent.~math.~{\bfseries 85} (1986), 337--371.

\bibitem{EHb}
D.~Eisenbud and J.~Harris,
\emph{The geometry of schemes},
Graduate Texts in Mathematics, vol.~197, 
Springer-Verlag, New York, 2000.

\bibitem{EM}
E.~Esteves and N.~Medeiros, 
\emph{Limits of ramification points of plane curves, using foliations}, In 
preparation.

\bibitem{F}
W.~Fulton,
\emph{Intersection Theory},
Ergebnisse der Mathematik und ihrer Grenzgebiete, 
series 3, vol.~2, Springer-Verlag, 
Berlin Heidelberg, 1984.

\bibitem{G}
A.~Grothendieck with J.~Dieudonn\'e,
\emph{\'Elem\'ents de G\'eom\'etrie 
Alg\'ebrique IV-4}, Inst.~Hautes \'Etudes Sci.~Publ. Math. {\bfseries 32} (1967)

\bibitem{No}
P.~Nogueira,
\emph{Limites de sistemas lineares em curvas de Gorenstein}, 
Doctor thesis, IMPA, 2003, available at 
http://www.preprint.impa.br/Shadows/SERIE\_ C/2008/68.html.

\end{thebibliography}
\end{document}